\documentclass[11pt]{article}

\usepackage{preprint}

\usepackage[
    backend=biber,
    style=alphabetic, 
    sortlocale=de_DE,
    citestyle=numeric-comp,
    natbib=true,
    url=false, 
    doi=true,
    eprint=false
]{biblatex}
\title{Degrees of Fano Quiver Moduli}
\author{Pengcheng Zhang}
\date{11.09.2026}

\begin{document}

\maketitle

\begin{abstract}
    This paper provides methods to compute degrees of Fano quiver moduli and gives explicit degree formulas for two classes of them. One such class is a series of toric Fano varieties that parametrize isomorphism classes of representations of the bipartite quiver \(K_{(2,q)}^m\) with dimension vector \(\underline{1}\), and the other class is the moduli space of \(q\)-point configurations on \(\mathbb{P}^1\). In the former case, we obtain a formula comprised of sums of binomial coefficients in \(q\) and \(m\). Its special case \(m=1\) turns out to be the same as the central MacMahon numbers (OEIS entry A177043). In the latter case, we obtained a formula over \(\mathbb{Q}\). These formulas are much more efficient than computing directly in Chow rings. 
\end{abstract}

\pagenumbering{arabic}

\tableofcontents

\section{Introduction}\label{Section:1}

Geometric Invariant Theory provides powerful tools to construct quotients of varieties by algebraic group actions. One interesting class of these quotients are quiver moduli, which parametrize isomorphism classes of stable representations of quivers with fixed dimension vectors. Algebraic, arithmetic, and geometric properties of quiver moduli are surveyed in \cite{Reineke08mo}. The Chow rings of quiver moduli were studied by King and Walter in \cite{kingwalter1995chow}, where they showed that Chern classes of universal bundles generate the Chow ring. In \cite{franzen2015chow} Franzen proved that they have explicit presentations. To make possible the computation of other invariants of quiver moduli, Belmans and Franzen described in \cite{belmans2024chow} the point class and Todd class for quiver moduli. 

This paper grew out of an idea suggested by my second advisor Markus Reineke. We compute degrees of Fano quiver moduli based on results mentioned above. A sufficient numerical criterion for a quiver moduli to be Fano is given in \cite{frRe2021fano}, as well as a formula for the anti-canonical class. By adapting \cite[Lemma 25]{franzen2015chow}, we provide a method in \zcref{Corollary:toric-red-degree} that helps compute degrees in the toric setting. 

Our main results are degree formulas for two types of Fano quiver moduli. One is the complete bipartite quiver \(\ktqm\) consisting of vertices \(v_1, v_2, w_1, \dots, w_q\) with \(q\geq 1\) being an odd integer and \(m\) arrows from \(v_i\) to \(w_j\), to which we assign the dimension vector \(\underline{1}\). \zcref{Proposition:k2qm-deg} computes the degree of the toric Fano variety \(\modulithetastable{\ktqm}{\underline{1}}\) as an explicit summation of products of binomial coefficients. The case \(m=1\) was studied by Blume and Hille in \cite{BlumeHille21}, where they showed that these quiver moduli recover Losev--Manin moduli spaces, which first appeared in \cite{LosevManin00} as a toric compactification of moduli spaces of \(n\)-pointed curves of genus 0. The degree formula for \(\modulithetastable{\ktqo}{\underline{1}}\) generates an integer sequence that depends on an odd integer \(q\). Computer experiments suggest that this sequence agrees with the central MacMahon numbers (OEIS entry A177043). This is stated as \zcref{Conjecture:deg-macmahon}, and an AI--assisted proof is given in \zcref{Proposition:macmahon-conj}. 

Another class of quivers we investigate is the subspace quiver \(S_q\) equipped with dimension vector \((1^q;2)\). The moduli space is precisely the space of ordered \(q\)-point stable configurations on \(\mathbb{P}^1\). The presentation of their Chow rings was given in \cite[Corollary 29]{franzen2015chow}, together with a comparison with that obtained via symplectic reduction. We compute its degree by applying \zcref{Corollary:toric-red-degree}. The strategy utilizes identities we found in \zcref{Proposition:k2qm-chow-eqs}, which convert both sides of the degree equation in \zcref{Corollary:toric-red-degree} to rational multiples of the same element. Finally, the degree formula is given in \zcref{Proposition:point-config-deg}. 

\section*{Acknowledgement}

I am deeply grateful to my advisors Jens Hornbostel and Markus Reineke for proposing this project and offering tremendous help in discussing and computing these formulas. I want to thank Hans Franzen for providing the idea of reducing computation to the toric case and an alternative formula for the point class in the toric case that greatly simplifies the computations in \zcref{Section:4}. I would also like to thank Gianni Petrella for modifying the julia package \cite{quivertools} to make verification of \zcref{Proposition:point-config-deg} possible for \(q\) up to 13. Furthermore, I want to thank Moritz Firsching for his help in formalizing \zcref{Conjecture:deg-macmahon} in Lean 4 and finding a proof using a LLM by DeepMind. 

This work is done with the funding by DFG Research Training Group 2240: \textit{Algebro-Geometric Methods in Algebra, Arithmetic and Topology}. 

\section{Recollection on quiver moduli}\label{Section:2}

Throughout this paper, we work over complex numbers \(\mathbb{C}\). We recall some basic notions and constructions from \cite{Reineke08mo}. 

A quiver \(Q\) is a directed graph consisting of a vertex set \(Q_0\) and an arrow set \(Q_1\), and equipped with a source map \(s\) and a target map \(t\) from \(Q_1\) to \(Q_0\). We restrict our focus to acyclic quivers in this paper, i.e. quivers without oriented cycles. 

A representation of a quiver is a collection \(M = ((M_i)_{i\in Q_0}, (M_{\alpha})_{\alpha\in Q_1})\) of \(\mathbb{C}\)-vector spaces \(M_i\) and \(\mathbb{C}\)-linear maps \(M_{\alpha}\) from \(M_{s(\alpha)}\) to \(M_{t(\alpha)}\). 
A dimension vector \(\underline{d}\) is a tuple in \(\mathbb{N}^{|Q_0|}\). We write \(\underline{d}(M)\) for the tuple \((\dim (M_i))_{i\in Q_0}\). Sometimes \(M\) will be referred to as a representation of \((Q, \underline{d})\). 
The Euler form of a quiver \(Q\) is the symmetric bilinear form on \(\mathbb{N}^{|Q_0|}\times \mathbb{N}^{|Q_0|}\) defined by 

\[
\langle \underline{d}, \underline{e}\rangle \coloneqq \sum_{i\in Q_0}d_ie_i - \sum_{\alpha\in Q_1}d_{s(\alpha)}e_{t(\alpha)}.
\]

Given \(\underline{d}\) for a quiver \(Q\), we define \(R(Q, d)\coloneqq \bigoplus_{\alpha\in Q_1}\mathrm{Hom}(M_{s(\alpha)}, M_{t(\alpha)})\), the space of all representations of \((Q, \underline{d})\). The group \(G_{\underline{d}} = \prod\limits_{i\in Q_0}GL(M_i)\) acts on \(R(Q, \underline{d})\) via base change:

\[
(g_i)_{i\in Q_0}\cdot (M_{\alpha})_{\alpha\in Q_1} \coloneqq (g_{t(\alpha)}M_{\alpha}g_{s(\alpha)}^{-1})_{\alpha\in Q_1}.
\]

Notice that the diagonally embedded scalars \(\mathbb{C}^*\hookrightarrow G_{\underline{d}}\) act trivially on \(R(Q, \underline{d})\). So the \(G_{\underline{d}}\)-action passes to a \(PG_{\underline{d}}\coloneqq G_{\underline{d}}/ \mathbb{C}^*\)-action. It is then natural to study the quotient of \(R(Q, \underline{d})\) by this action, which in general does not carry the structure of a complex variety. In order to find a nice variety for the quotient, we need the notion of stability. 

A stability parameter for \((Q, \underline{d})\) is a \(\mathbb{Z}\)-linear form \(\theta\) on \(\mathbb{Z}^{|Q_0|}\) such that \(\theta(\underline{d}) = 0\). 

\begin{definition}
    Given \((Q, \underline{d})\) and a stability parameter \(\theta\), a representation \(M\) of \(Q\) with dimension vector \(\underline{d}\) is called \(\theta\)-semistable (resp. stable) if \(\theta (\underline{d}(M')) \leq 0\) (resp. \(\theta(\underline{d}(M')) < 0\)) for all non-zero proper subrepresentations \(M'\subset M\).
\end{definition}

There is a natural inclusion of open subsets of \(\theta\)-stable and \(\theta\)-semistable locus of \(R(Q, \underline{d})\):

\[
R^{\theta\text{-st}}(Q, \underline{d})\subseteq R^{\theta\text{-sst}}(Q, \underline{d})\subseteq R(Q, \underline{d})
\]

We will always assume \(\underline{d}\) is \(\theta\)-coprime, i.e. \(\theta(\underline{d}(M'))\neq 0, \text{ for all non-zero proper } M'\subset M\). As a result, the first inclusion is an identity. By Geometric Invariant Theory (see \cite[Section 3.5]{Reineke08mo} for more details), there exists a smooth projective complex variety \(M^{\theta\text{-st}}(Q, \underline{d})\) of dimension \(1 - \langle \underline{d}, \underline{d}\rangle\eqqcolon N\) if not empty, parametrizing all isomorphism classes of \(\theta\)-stable representations of \(Q\) with dimension vector \(\underline{d}\). 
We will write \(X\) for the variety \(M^{\theta\text{-st}}(Q, \underline{d})\) in some subsections when it is clear from the context. 

In \cite{kingwalter1995chow}, King and Walter described universal bundles \(\mathcal{V}_i\) such that \(\mathrm{rank} (\mathcal{V}_i)=d_i, i\in Q_0\) on \(X\) and showed that \(\mathrm{CH}^*(X)\) is generated by \(c_j(\mathcal{V}_i), j = 1, \dots, d_i, i\in Q_0\). Franzen gave an explicit presentation of \(\mathrm{CH}_{\mathbb{Q}}^*(X)\) in \cite[Proposition 14]{franzen2015chow}. Although not explicitly proven, the presentation should also hold for \(\mathbb{Z}\) coefficients because the proof relies on \cite[Theorem 10.8]{DanilovToric78}, which does not require \(\mathbb{Q}\) coefficients. All computations in this paper will be done with coefficients in \(\mathbb{Q}\), so we omit the subscript \(\mathbb{Q}\) when writing Chow rings. The following presentation is taken from \cite[Section 2]{belmans2024chow}. 

Let \(a\in \mathrm{Hom}_{\mathbb{Z}}(\mathbb{Z}^{|Q_0|}, \mathbb{Z})\) such that \(a(\underline{d}) = \sum_{i\in Q_0}a_i d_i = 1\), which is possible under the assumption that \(\underline{d}\) is \(\theta\)-coprime. Define the permutation group \(W_{\underline{d}} \coloneqq \prod_{i\in Q_0}\text{Sym}_{d_i}\) and polynomial rings 

\[
R \coloneqq \bigotimes_{i\in Q_0}\mathbb{Q}[\xi_{i, 1}, \dots, \xi_{i, d_i}], A\coloneqq R^{W_{\underline{d}}} = \bigotimes_{i\in Q_0}\mathbb{Q}[x_{i, 1}, \dots, x_{i, d_i}],
\]

where \(A\) is the subring of \(W_{\underline{d}}\)-invariants and is generated by \(x_{i,j} = e_j(\xi_{i, 1}, \dots, \xi_{i, d_i})\), the \(j\)-th elementary symmetric polynomial in variables \(\xi_{i, 1}, \dots, \xi_{i, d_i}\) for all \(i\in Q_0, j = 1, \dots, d_i\). Here \(W_{\underline{d}}\) acts via permuting variables. Namely, \(\sigma \cdot f((\xi_{i, j})_{i\in Q_0, j = 1, \dots, d_1}) = f((\sigma_i(\xi_{i, j}))_{i\in Q_0, j = 1, \dots, d_i}) \) for \(\sigma = (\sigma_i)_{i\in Q_0}\in W_{\underline{d}}\). 

There is a natural surjection \(\rho\colon R\rightarrow A\) called the symmetrization map defined as follows:

\begin{equation}
    \rho: R \rightarrow A, f \mapsto \frac{1}{\Delta}\sum_{\sigma\in W_{\underline{d}}}\text{sign}(\sigma)\sigma(f),\label{Map:symmetrization}
\end{equation}

where \(\Delta = \prod_{i\in Q_0}\prod_{1\leqslant k < l\leqslant d_i}(\xi_{i, l} - \xi_{i, k})\) is called the discriminant of the \(W_{\underline{d}}\)-action. It is well-known that every anti-invariant polynomial in \(R\) under the \(W_{\underline{d}}\)-action can be written as a product \(\Delta\cdot g\) for some \(g\in A\) (see \cite[Chapter \Romannum{5}, \S 5, Proposition 5]{bourbaki1994lie} for more details). The ring \(R\) has rank \(\# W_{\underline{d}}\) as an \(A\)-module. Moreover, \(\rho\) restricted on \(R^{\text{anti}}\) is an isomorphism of \(\mathbb{Q}W_{\underline{d}}\)-modules with decreasing degree \(\delta\coloneqq \deg \Delta\), where \(R^{\text{anti}}\) is the anti-invariants of \(R\) under the \(W_{\underline{d}}\)-action. 

\begin{equation}
    \rho: R^{\text{anti}}\rightarrow A, \Delta \cdot g\mapsto g\label{Map:restricted-rho}
\end{equation}

Finally, we define the tautological ideal \(I_{\text{taut}}\) in \(R\) that is generated by polynomials of the form 

\[
f_{d'} = \prod_{\alpha\in Q_1}\prod_{k = 1}^{d'_{s(\alpha)}}\prod_{l = d'_{t(\alpha)} + 1}^{d_{t(\alpha)}}(\xi_{t(\alpha), l} - \xi_{s(\alpha), k})
\]

for all \(\underline{d}'\) such that \(\theta(\underline{d}') > 0\), and the linear ideal \(I_{\text{lin}} = \langle \sum_{i\in Q_0}a_i x_{i, 1}\rangle\) of \(A\). Then \cite[Proposition 14]{franzen2015chow} tells us the following. 

\begin{theorem}\label{Theorem:chow-quiver-pre}
    The \(\mathbb{Q}\)-algebra morphism \(A\rightarrow \mathrm{CH}^*(X)\) given by \(x_{i,j}\mapsto c_j(\mathcal{V}_i)\) is surjective and the kernel is given by \(I_{\text{lin}} + \rho(I_{\text{taut}})\).

\end{theorem}

Another tool we will use in our computation is the following identification from \cite[Lemma 2.2]{ReiKronecker}.

\begin{lemma}\label{Lemma:quiver-op-iso}
    There exists an isomorphism 

    \[
    M^{\theta\text{-sst}}(Q, \underline{d})\simeq M^{-\theta\text{-sst}}(Q^{\text{op}}, \underline{d}),
    \]

    where \(Q^{\text{op}}\) has the same vertices as \(Q\), but all arrows are reversed. 
\end{lemma}

\begin{remark}
    Since we always assume \(\theta\)-coprimality, this isomorphism persists when restricted to the moduli spaces of \(\theta\)-stable representations. 
\end{remark}

\section{Degree computation and toric reduction}\label{Section:3}

In this section, we define the degree of a Fano quiver moduli, and give methods to compute the degree in both toric and non-toric cases. 

\subsection{Definition of the degree}

A dimension vector \(\underline{d}\) is said to be \(\theta\)-strongly amply stable if \(\langle \underline{e}, \underline{d} - \underline{e}\rangle < -2\) for all non-zero proper subvectors \(\underline{e} < \underline{d}\) with \(\theta(\underline{e}) > 0\). 

We consider the canonical stability parameter defined as 

\[
\tcan\coloneqq \langle \underline{d}, \_\rangle - \langle \_, \underline{d}\rangle.
\]

We write \(\Theta_i\) for \(\tcan(\underline{i})\), where \(i\in Q_0\) and \(\underline{i}\) is the unit dimension vector that has 1 at the \(i\)-th entry and 0 elsewhere. The following result is \cite[Theorem 4.3]{frRe2021fano}.

\begin{theorem}\label{Theorem:anti-can-class-formula}
    Assume \(\underline{d}\) is \(\tcan\)-strongly amply stable. Then the moduli space \(X\coloneqq M^{\tcan\text{-st}}(Q, \underline{d})\) is a smooth projective Fano variety of dimension \(N\coloneqq 1 - \langle \underline{d}, \underline{d}\rangle\). Furthermore, the anti-canonical class \(-K_X\coloneqq c_1(\det(\mathcal{T}))\), where \(\mathcal{T}\) is the tangent bundle on \(X\), is computed as 
    
    \begin{equation}
        -K_X = \sum_{i\in Q_0}(-\Theta_i) c_1(\mathcal{V}_i).
    \end{equation}
\end{theorem}

For quiver moduli in our setting, all points are rationally equivalent. Then another ingredient for the computation is the point class formula given in \cite[Theorem B]{belmans2024chow},

\begin{equation}
    [pt] = \frac{\prod_{\alpha\in Q_1}c(\mathcal{V}_{t(\alpha)})^{d_{s(\alpha)}}}{\prod_{i\in Q_0}c(\mathcal{V}_i)^{d_i}}\big|_N.\label{Equation:quiver-moduli-pt}
\end{equation}

Here \(|_N\) means taking the homogeneous degree \(N\) part of the expression to which it is applied. This can be done by formally expand the inverse of total Chern classes in the denominator, which is well-established in \cite[Section 3.2]{Fulton98}. 

\begin{definition}
    Assume the conditions of \zcref{Theorem:anti-can-class-formula}. The degree of the anti-canonical class \(-K_X\) is the unique integer \(D\) such that the equation

    \begin{equation}
        (-K_X)^N = D[pt]\label{Equation:Fano-deg-def}
    \end{equation}

    holds in \(CH^N(X)\). 
\end{definition}

The number \(D\) is indeed the degree of \(X\) with respect to the anti-canonical embedding. Let \(\phi:X\hookrightarrow \mathbb{P}^n\) be the anti-canonical embedding of \(X\) into some \(\mathbb{P}^n\). The degree \(\deg(\phi)\) with respect to \(\phi\) is the cardinality of the intersection of \(X\) with \(N\) hyperplanes in general position, i.e. \(\deg(\phi)[pt]_{\mathbb{P}^n} = \phi_*([X])\cdot h^N\) holds in \(\mathrm{CH}^{n}(\mathbb{P}^n)\), where \(h\coloneqq c_1(\mathcal{O}_{\mathbb{P}^n}(1))\). Notice that \(-K_X = \phi^*(h)\). By \cite[Proposition 2.5(c)]{Fulton98}, we have an equation in $\mathrm{CH}^n(\mathbb{P}^n)$

\begin{align*}
    D[pt]_{\mathbb{P}^n} = \phi_*(D[pt]_X) &= \phi_*((-K_X)^N\cdot [X]) \\ 
    &= h\cdot \phi_*((-K_X)^{N-1}\cdot [X]) = \cdots  = h^N\cdot \phi_*([X]) = \deg(\phi)[pt]_{\mathbb{P}^n}.
\end{align*}

\subsection{The point class in the toric case}

When the dimension vector is \(\underline{1} = (1)_{i\in Q_0}\), the moduli space \(X \coloneqq M^{\tcan\text{-st}}(Q, \underline{1})\) is a toric variety, on which \(\mathbb{G}_m^{|Q_1|} / PG_{\underline{1}}\) acts with a dense orbit (an argument is given in \cite[Section 4.1]{franzen2015chow}). Recall from \cite{franzen2015chow} that the support \(\text{supp}(M)\) of \(M\in R(Q, \underline{1})\) is \(\{\alpha\in Q_1~|~M_{\alpha}\neq 0\}\), and a subset \(J\subset Q_1\) is called \(\theta\)-semistable (resp. \(\theta\)-stable) if there exists a \(\theta\)-semistable (resp. \(\theta\)-stable) representation \(M\) of \((Q, \underline{1})\) such that \(\text{supp}(M) = J\). We say \(J\subset Q_1\) is a spanning tree if the underlying undirected graph of \((Q_0, J)\) is connected and acyclic. In a private communication, H. Franzen provided an alternative formula for the point class in the toric case and a sketch of the proof. We state and prove it below. 

\begin{theorem}\label{Theorem:point-class-via-spanning-tree}
    Assume \((Q, \underline{1})\) satisfies the conditions of \zcref{Theorem:anti-can-class-formula} and \(\theta\) is a stability parameter such that \(R^{\theta\text{-st}}(Q, \underline{1})\) is non-empty. There exists a \(\theta\)-stable spanning tree \(\Gamma\) of \(Q\), and the point class of \(X\) is 

    \begin{equation}\label{Equation:toric-pt}
        [pt] = \prod_{\alpha\in Q_1\backslash \Gamma}(c_1(\mathcal{L}_{t(\alpha)}) - c_1(\mathcal{L}_{s(\alpha)}))
    \end{equation}
\end{theorem}

\begin{proof}
    Consider the torus action \((\mathbb{C}^*)^{|Q_1|}\) on \(X\) given by 

    \begin{align*}
        (\mathbb{C}^*)^{|Q_1|} \times X&\rightarrow X\\ 
        (c_{\alpha})_{\alpha\in Q_1}\cdot [(M_{\alpha})_{\alpha\in Q_1}] &\mapsto [(c_{\alpha}M_{\alpha})_{\alpha\in Q_1}].
    \end{align*}
    
    It has at least one fixed point \([M]\) whose support is a \(\theta\)-stable spanning tree. 

    This claim can be proven by contradiction. If \(J_M\) is not connected, \(M\) is decomposable, hence not stable. Suppose the underlying undirected graph of \(J_M\) has a cycle \(\alpha_1, \dots, \alpha_s\) of length \(s\geq 2\) (this includes multiple arrows between vertices). The above torus action on \(M\) is the same as \((\mathbb{C}^*)^{|J_M|}\)-action on \(M\), where \(|J_M|\geq |Q_0|\) in this case. Then this action cannot fix \([M]\) since it has a positive dimensional orbit. 
    
    We fix such a \(\tcan\)-stable spanning tree \(\Gamma\). Let \(M_{\Gamma}\) be the representation in \(R(Q, \underline{1})\) given by

    \[
    M_{\alpha} = \begin{cases}
        1, & \text{if \(\alpha\in \Gamma\),}\\ 
        0, & \text{if \(\alpha\notin \Gamma\).}
    \end{cases}
    \]

    We write \(w = [M_{\Gamma}]\in X\). Let \(\mathcal{L}_i, i\in Q_0\) be the universal line bundles on \(X\). Consider the bundle below determined by a spanning tree \(\Gamma\):

    \[
    \mathcal{E}_{\Gamma} = \bigoplus_{\alpha\in Q_1\backslash\Gamma}\mathcal{L}_{s(\alpha)}^{\vee}\otimes\mathcal{L}_{t(\alpha)},
    \]

    which has rank \(|Q_1| - |Q_0| + 1 = N\). It has a natural section \(f\in H^0(X, \mathcal{E}_{\Gamma})\) given by 

    \[
    f(y) = (M'_{\alpha})_{\alpha\in Q_1\backslash\Gamma}\in \mathcal{E}_{\Gamma}|_y = \bigoplus_{\alpha\in Q_1\backslash \Gamma}\mathrm{Hom}(M'_{s(\alpha)}, M'_{t(\alpha)})
    \]

    for \(y = [M']\in X\). We can easily verify that \(f(w) = 0\). Conversely, given any \(y = [M']\in X\) such that \(f(y) = 0\), we have \(M'_{\alpha} = 0\) for \(\alpha\in Q_1\backslash \Gamma\). Since \(M'\) is \(\theta\)-stable, the support of \(M'\) must be \(\Gamma\). Hence, \(M'\simeq M\). This way, we have shown that \(Z(s) = \{w\}\). Since dimension vectors considered in our setting are all \(\theta\)-coprime, the quiver moduli \(X\) is rational by \cite[Theorem 6.4]{Schofield01}. Hence, \(\mathrm{CH}^N(X)\) has rank 1 and is generated by any point. By \cite[Proposition 14.1 (a)]{Fulton98}, we have 

    \[
    [pt] = [w] = c_N(\mathcal{E}_{\Gamma}) = \prod_{\alpha\in Q_1\backslash \Gamma}c_1(\mathcal{L}_{s(\alpha)}^{\vee}\otimes \mathcal{L}_{t(\alpha)}) = \prod_{\alpha\in Q_1\backslash \Gamma}(c_1(\mathcal{L}_{t(\alpha)}) - c_1(\mathcal{L}_{s(\alpha)}))
    \]
\end{proof}

\subsection{Toric reduction}

One advantage we have when computing the degree in the toric case is that the symmetrization map \(\rho\) is much simpler since the permutation group \(W_{\underline{1}}\) is trivial. We shall see that the degree computation of non-toric quiver moduli can be reduced to a computation in a toric setting. 

Recall from \cite[Section 4.2]{franzen2015chow} the following notion.

\begin{definition}
    Let \(Q\) be an acyclic quiver and \(\underline{d}\) a dimension vector. The covering quiver \(\tilde{Q}\) of \(Q\) with respect to \(\underline{d}\) is the pair \((\tilde{Q}_0, \tilde{Q}_1)\) given by

    \begin{align*}
        \tilde{Q}_0 &\coloneqq \{(i, j)~|~i\in Q_0, 1\leq j\leq d_i\}\\ 
        \tilde{Q}_1 &\coloneqq \{(\alpha, j_1, j_2)~|~\alpha\in Q_1, 1\leq j_1\leq d_{s(\alpha)}, 1\leq j_2\leq d_{t(\alpha)}\}.
    \end{align*}
\end{definition}

Heuristically, we split up the \(i\)-th vertex in \(Q_0\) into \(d_i\) new vertices in \(\tilde{Q}_0\) and preserve all arrows adjacent to it. Consequently, we have \(R(Q, \underline{d}) = R(\tilde{Q}, \underline{1})\). Denote by \(\tilde{\Theta}^{\text{can}}\) the canonical stability parameter associated to \(\tilde{Q}\) with \(\underline{1}\). Direct computation shows that \(\tilde{\Theta}_{(i, j)} = \Theta_i, \text{ for all } (i, j)\in \tilde{Q}_0\). We thus have an identification \(R^{\tcan\text{-st}}(Q, \underline{d}) = R^{\tilde{\Theta}^{\text{can}}\text{-st}}(\tilde{Q}, \underline{1})\), on which \(G_{\underline{d}}\) and \(G_{\underline{1}}\) act respectively. After taking quotients, we have \(X\coloneqq \modulithetastable{Q}{\underline{d}}\), and \(\tilde{X}\coloneqq \modulitildethetastable{\tilde{Q}}{\underline{1}}\). 

To compare the Chow rings of \(X\) and \(\tilde{X}\), we start with the ring \(R = \bigotimes_{i\in Q_0}\mathbb{Q}[\xi_{i, 1}, \dots, \xi_{i, d_i}]\) introduced in \zcref{Section:2} and compute the invariants under \(W_{\underline{d}}\) and \(W_{\underline{1}}\) actions respectively. 

\[\begin{tikzcd}
	& R & \\
	{R^{W_{\underline{d}}}} && {R^{W_{\underline{1}}}}
	\arrow["{\rho^{W_{\underline{d}}}}"', from=1-2, to=2-1]
	\arrow["{\rho^{W_{\underline{1}}}}", from=1-2, to=2-3]
\end{tikzcd}\]

The permutation group \(W_{\underline{1}}\) is trivial, so we simply have \(\rho^{W_{\underline{1}}}: R\simeq R^{W_{\underline{1}}}\) as \(\mathbb{Q}\)-modules. We will just write \(\rho\) for \(\rho^{W_{\underline{d}}}\) below. Choose \(a\in \mathrm{Hom}_{\mathbb{Z}}(\mathbb{Z}^{|Q_0|}, \mathbb{Z})\) such that \(a({\underline{d}}) = 1\). Denote \(a_i\coloneqq a(\underline{i}) \) for \(i\in Q_0\). Let \(\tilde{a}\in \mathrm{Hom}_{\mathbb{Z}}(\mathbb{Z}^{|\tilde{Q}_0|}, \mathbb{Z})\) such that \(\tilde{a}(\underline{(i, j)}) = a_i\). Then \(\tilde{a}(\underline{1}) = a(\underline{d}) = 1\). We write \(\tilde{I}_{\text{lin}}\) for the ideal \(\langle \sum_{(i, j)\in \tilde{Q}_0}\tilde{a}_{(i, j)} \xi_{i,j}\rangle \) in \(R\). Combined with the map in \zcref{Theorem:chow-quiver-pre}, we have a diagram modified from that in \cite[Section 4.2]{franzen2015chow}, where the two middle terms are denoted by \(C\) and \(A\), and the ideals in these two rows are \(\mathfrak{c}, \mathfrak{a}\) in the original diagram. We also added the induced map \(\bar{\rho}^{-1}\) on the quotient which is not present in op.cit. 

\[\begin{tikzcd}
	0 & {\tilde{I}_{\text{lin}} + I_{\text{taut}}} & {\bigotimes\limits_{i\in Q_0}\mathbb{Q}[\xi_{i, 1}, \dots, \xi_{i, d_i}]} & {\mathrm{CH}^*(\tilde{X})} & 0 \\
	0 & {I_{\text{lin}} + \rho(I_{\text{taut}})} & {\bigotimes\limits_{i\in Q_0}\mathbb{Q}[x_{i, 1}, \dots, x_{i, d_i}]} & {\mathrm{CH}^*(X)} & 0
	\arrow[from=1-1, to=1-2]
	\arrow[from=1-2, to=1-3]
	\arrow[from=1-3, to=1-4]
	\arrow[from=1-4, to=1-5]
	\arrow[from=2-1, to=2-2]
	\arrow["\rho^{-1}", from=2-2, to=1-2]
	\arrow[from=2-2, to=2-3]
	\arrow["\rho^{-1}", from=2-3, to=1-3]
	\arrow[from=2-3, to=2-4]
	\arrow[from=2-4, to=2-5]
    \arrow[dashed, "\bar{\rho}^{-1}", from=2-4, to=1-4]
\end{tikzcd}\]

Both rows are exact, and the map \(\rho^{-1}\) is the inverse of \(\rho\) restricted to \(R^{\text{anti}}\), which is the isomorphism \ref{Map:restricted-rho} of \(\mathbb{Q}W_{\underline{d}}\)-modules considered above. Written explicitly,

\begin{equation}
    \rho^{-1}(g((x_{i,j})_{i\in Q_0, j = 1, \dots, d_i})) = \Delta g((e_j(\xi_{i, 1}, \dots, \xi_{i, d_i}))_{(i, j)\in \tilde{Q}_0}).
    \label{Equation:rho-bar-inv-def}
\end{equation}

Then by \cite[Lemma 25]{franzen2015chow}, we know that \(\rho^{-1}\) passes to an isomorphism \(\bar{\rho}^{-1}\) on the quotient \(\mathrm{CH}^*(X)\rightarrow \mathrm{CH}^*(\tilde{X})^{\text{anti}}\). Notice that for \(-K_X = \sum_{i\in Q_0}-\tcan_ix_{i, 1}\), 

\begin{align*}
    \bar{\rho}^{-1}((-K_X)^N) &= \Delta \left( \sum_{i\in Q_0}\left(\sum_{\alpha:j\rightarrow i}d_j - \sum_{\beta:i\rightarrow k}d_k\right)x_{i, 1} \right)^N\\ 
    &= \Delta \left(\sum_{i\in Q_0}\left(\sum_{\alpha:j\rightarrow i}d_j - \sum_{\beta:i\rightarrow k}d_k\right)\sum_{j = 1}^{d_i}\xi_{i,j}\right)^N\\ 
    &= \Delta \left( \sum_{(i,j)\in \tilde{Q}_0}-\tilde{\Theta}_{(i,j)}\xi_{i,j} \right)^N\\
    &= \Delta (-K_{\tilde{X}})^N.
\end{align*}

Therefore, after applying \(\bar{\rho}^{-1}\) to \zcref{Equation:Fano-deg-def} we obtain the following corollary of \cite[Lemma 25]{franzen2015chow}.

\begin{corollary}
    \label{Corollary:toric-red-degree}
    Let \(Q\) be an acyclic quiver with dimension vector \(\underline{d}\) and assume \(\underline{d}\) is \(\tcan\)-coprime and \(\tcan\)-amply stable. Then the degree of \(-K_X\) of the Fano variety \(X = M^{\tcan\text{-st}}(Q, \underline{d})\) can be computed by the following equation in \(\displaystyle CH^{N + \delta}_{\mathbb{Q}}(\tilde{X})\):

    \[
    (-K_{\tilde{X}})^N\Delta = D\bar{\rho}^{-1}([pt]).
    \]
\end{corollary}

\section{Bipartite quiver moduli and point configuration spaces on \texorpdfstring{\(\mathbb{P}^1\)}{P1}}\label{Section:4}

In this section, we study complete bipartite quivers \(\ktqm\) with vertices \(\{v_1, v_2, w_1, \dots, w_q\}\) and \(m\) arrows from \(v_i\) to \(w_j\) for all \(i = 1, 2, j = 1, \dots, q\). The quiver moduli \(\modulithetastable{\ktqm}{\underline{1}}\) is a toric Fano variety of dimension \(N = 2qm - q - 1\). We will compute its degree in \zcref{Subsection:4-1}, and describe some identities in its Chow groups in \zcref{Subsection:4-2}. Finally in \zcref{Subsection:4-3}, we will apply these identities to the computation of the degree of space of point configurations on \(\mathbb{P}^1\). 

\subsection{Degree formula for \texorpdfstring{\(\modulithetastable{\ktqm}{\underline{1}}\)}{M(K2qm, 1)}}\label{Subsection:4-1}

Following \zcref{Theorem:chow-quiver-pre}, we see that 

\[
\mathrm{CH}^*(\modulithetastable{\ktqm}{\underline{1}})\simeq \mathbb{Q}[x_1, x_2;y_1, \dots, y_q] / (I_{\text{lin}} + I_{\text{taut}}),
\]

where \(x_i\coloneqq c_1(\mathcal{L}_{v_i}), y_i\coloneqq c_1(\mathcal{L}_{w_j})\), and the ideal \(I_{\text{lin}} + I_{\text{taut}}\) is generated by four types of relations:

\begin{itemize}
    \itemsep=0pt
    \parskip=0pt
    \parsep=0pt
    
    \item Linear relation: \(a_1(x_1 + x_2) + a_2(y_1 + \cdots y_q) = 0\) for some \(a_1, a_2\in \mathbb{Z}\) satisfying \(2a_1 + qa_2 = 1\). Here \(a\) is chosen this way to be compatible with the \(\text{Sym}_2\times \text{Sym}_q\)-action on \(K_{(2,q)}^m\). 

    \item Tautological relation (1): \(T_j^{(1)} = (y_j - x_1)^m(y_j - x_2)^m\) for $j = 1, \dots, q$.

    \item Tautological relation (2): \(T_I^{(2)} = \prod\limits_{j\in I}(y_j - x_1)^m\) for \(I\subset \{1, \dots, q\}\) such that \(|I| = \frac{q + 1}{2}\). 

    \item Tautological relation (3): \( T_I^{(3)} = \prod\limits_{j\in I}(y_j - x_2)^m\) for \(I\subset \{1, \dots, q\}\) such that \(|I| = \frac{q + 1}{2}\). 
\end{itemize}

In fact, our computation does not depend on the choice of \(a\). So we won't specify it below. 

By \zcref{Theorem:anti-can-class-formula}, we have 

\[
-K_X = -\Theta_{v_1}x_1 - \Theta_{v_2}x_2 + \Theta_{w_1}y_1 + \cdots + \Theta_{w_q}y_q = m\sum_{j = 1}^q[(y_j - x_1) + (y_j - x_2)].
\]

Now we come to the point class. First we find all possible \(\tcan\)-stable spanning trees. Consider a spanning tree \(\Gamma(J_1, J_2)\) such that \(J_1\cup J_2 = \{w_1, \dots, w_q\}, |J_1| = |J_2| = \frac{q+1}{2}, |J_1\cap J_2| = 1\), and the arrows are those from \(v_1\) to \(J_1\) and from \(v_2\) to \(J_2\). 

\begin{lemma}
    Given a spanning tree \(\Gamma(J_1, J_2)\), define a representation \(M\in R(K_{(2, q)}^m, \underline{1})\) by 

    \[
    M_{\alpha} = \begin{cases}
        1, & \alpha\in \Gamma(J_1, J_2),\\ 
        0, & \text{else}
    \end{cases}
    \]

    Then \(M\) is a \(\tcan\)-stable representation of \(K_{(2, q)}^m\), hence \(\Gamma(J_1, J_2)\) is a \(\tcan\)-stable spanning tree. Moreover, there are no other \(\tcan\)-stable spanning trees. 
\end{lemma}

\begin{proof}
    The proof is simply verifying the criterion given in \cite[Lemma 19]{franzen2015chow}. 
\end{proof}

Therefore, we have a point class formula for any given \(\tcan\)-stable spanning tree \(\Gamma(J_1, J_2)\),

\begin{align}
    [pt] &= \prod_{\alpha\notin \Gamma(J_1, J_2)}(c_1(\mathcal{L}_{t(\alpha)}) - c_1(\mathcal{L}_{s(\alpha)}))\notag\\ 
    &= \prod_{j\in J_1\backslash J_2}(y_j - x_1)^m(y_j - x_2)^{m-1}\prod_{j\in J_2\backslash J_1}(y_j - x_1)^{m-1}(y_j - x_2)^m (y_k - x_1)^{m-1}(y_k - x_2)^{m-1},\label{Equation:k2qm-pt}
\end{align}

where \(\{k\} = J_1\cap J_2\). On the other hand, we compute \((-K_X)^N\):

\[
(-K_X)^{N} = m^N\sum_{\substack{r_1 + \cdots + r_q + s_1 + \cdots + s_q = N\\ r_1, \dots, r_q, s_1, \dots, s_q\geq 0}}\frac{N!}{r_1!\cdots r_q!s_1!\cdots s_q!}\prod_{j = 1}^q(y_j - x_1)^{r_j}(y_j - x_2)^{s_j}.
\]

The degree \(D = \frac{(-K_X)^N}{[pt]}\) is an integer that depends on \(q\) and \(m\), thus will be written \(D(q, m)\). We introduce some notation to simplify the presentation of subsequent computations.

\begin{table}[H]
    \centering
    {
        \renewcommand{\arraystretch}{1.6}
        \begin{tabular}{|c|c|}
            \hline 
            \(P(N, 2q)\) & the set of partitions of \(N\) into \(2q\) non-negative integers\\ 
            \hline
            \(\boldpartitions{r}{s}\) & an ordered partition of the form \([r_1, \dots, r_q, s_1, \dots, s_q]\)\\ 
            \hline
            \(J_r, J_s\) & \makecell{subsets of \(\{1, \dots, q\}\) determined by \(\boldpartitions{r}{s}\), \\ \(J_r \coloneqq \{j~|~r_j\geq m\}, J_s \coloneqq \{j~|~s_j\geq m\}\)}\\ 
            \hline
            \(j_r, j_s\) & cardinalities of \(J_r\) and \(J_s\), respectively \\ 
            \hline
            \(\binom{N}{\boldpartitions{r}{s}}\) & multinomial coefficient \(\frac{N!}{r_1!\cdots r_q!s_1!\cdots s_q!}\)\\ 
            \hline 
            \(p_{\boldpartitions{r}{s}}(\underline{x}, \underline{y})\) & the product \(\prod_{j = 1}^q(y_j - x_1)^{r_j}\prod_{j = 1}^q(y_j - x_2)^{s_j}\)\\
            \hline
            \(\displaystyle \sum_{J}^{k}\) & a summation over all subsets \(J\) of \(\{1, \dots, q\}\) of size \(k\)\\ 
            \hline 
        \end{tabular}
    }
\end{table}

\begin{proposition}\label{Proposition:k2qm-deg}
    The degree formula of the Fano variety \(\modulithetastable{\ktqm}{\underline{1}}\) is 

    \begin{align}
        D(q,m) =& m^{N}\sum_{\boldpartitions{r}{s}\in \mathrm{nP}(N, 2q)}\binom{N}{[\mathbf{r}, \mathbf{s}]}\prod_{j\in J_r}\binom{r_j - m}{m-1-s_j}\prod_{j\in J_s}(-1)^{r_j + s_j + 1}\binom{s_j - m}{m-1-r_j}\notag\\ 
        &\times(-1)^{\frac{q-1}{2}-j_s}\binom{q-1-j_r-j_s}{\frac{q-1}{2}-j_r}\prod_{j\notin (J_r\cup J_s)}(-1)^{m-1-s_j}\binom{2m-2-r_j-s_j}{m-1-s_j},
    \end{align}

    where \(\mathrm{nP}(N, 2q)\) is a subset of \(P(N, 2q)\) that consists of partitions in \(P(N, 2q)\) satisfying \(J_r\cap J_s = \emptyset, j_r\leq\frac{q-1}{2},j_s\leq\frac{q-1}{2}\), and \(r_j + s_j\geq 2m-1, \text{ for all } j\in J_r\cup J_s\). 
\end{proposition}

\begin{proof}
    We compute integers \(D_{\boldpartitions{r}{s}}\) such that \(p_{\boldpartitions{r}{s}}(\underline{x}, \underline{y}) = D_{\boldpartitions{r}{s}} [pt]\). The way to achieve this is to reduce each \(p_{\boldpartitions{r}{s}}(\underline{x}, \underline{y})\) via tautological relations to the form of \zcref{Equation:k2qm-pt} for some \(J_1, J_2\), and \(k\in J_1\cap J_2\). We analyze \(\boldpartitions{r}{s}\) in three cases, each of which is considered when all previous cases fail. 

    \textbf{Case 1:} \(J_r\cap J_s\neq \emptyset\).

    There exists at least one \(j\in J_r\cap J_s\), i.e. \(r_j, s_j\geq m\). The tautological relation \(T_j^{(1)}\) divides \(p_{\boldpartitions{r}{s}}(\underline{x}, \underline{y})\), hence \(D_{\boldpartitions{r}{s}} = 0\). 

    \textbf{Case 2:} \(j_r > \frac{q-1}{2}\) or \(j_s > \frac{q-1}{2}\).

    This means the tautological relation \(T_{J_r}^{(2)}\) or \(T_{J_s}^{(3)}\) divides \(p_{\boldpartitions{r}{s}}(\underline{x}, \underline{y})\), hence \(D_{\boldpartitions{r}{s}} = 0\). 

    \textbf{Case 3:} \(r_j + s_j < 2m-1\) for some \(j\in J_r\cup J_s\). 

    Without loss of generality, we only consider the case \(j\in J_r\) and \(r_j + s_j < 2m-1\). By applying the substitution \(y_j - x_1 = (y_j - x_2) + (x_2 - x_1)\) to the factor \((y_j - x_1)^{r_j - m}\) and expanding it, we move extra exponents on \((y_j - x_1)\) to \((x_2 - x_1)\) and \((y_j - x_2)\). After performing such expansions for all \(j\in J_r\cup J_s\), \(p_{\boldpartitions{r}{s}}(\underline{x}, \underline{y})\) becomes a \(\mathbb{Z}\)-linear combination of terms of the form \(p_{[\mathbf{r}', \mathbf{s}']}(\underline{x}, \underline{y})(x_2-x_1)^{e(\boldpartitions{r'}{s'})}\), where each \([\mathbf{r}', \mathbf{s}']\) satisfies 
    
    \[
    \begin{cases}
        r'_j = m, s_j\leq s'_j\leq m-1 & \text{if }j\in J_r\\ 
        r_j\leq r'_j\leq m-1, s'_j = m & \text{if }j\in J_s\\ 
        r'_j = r_j, s'_j = s_j & \text{if }j\in \{1, \dots, q\}\backslash (J_r\cup J_s),
    \end{cases}
    \]

    and \(e(\boldpartitions{r'}{s'}) = \sum_{j\in J_r}(r_j + s_j - m - s'_j) + \sum_{j\in J_s}(r_j + s_j - m - r'_j)\). 

    For every partition \(\boldpartitions{r'}{s'}\), we have \(J_{r'} = J_r\) and \(J_{s'} = J_s\). Remember that we try to convert \(p_{\boldpartitions{r}{s}}(\underline{x}, \underline{y})\) into a multiple of \([pt]\). In other words, we enlarge \(J_r\) and \(J_s\) until they both have size \(\frac{q-1}{2}\) and \(r_j + s_j = 2m-1, \text{ for all } j\in J_r\cup J_s\), and leave one index \(k\notin J_r\cup J_s\) at which we have a factor \((y_k - x_1)^{m-1}(y_k - x_2)^{m-1}\). 

    Our strategy is to first choose some \(k\notin J_r\cup J_s\) and expand the factor \((x_2 - x_1)^{2m-2-r_k-s_k}\) in \((x_2-x_1)^{e(\boldpartitions{r'}{s'})}\) as
    
    \begin{align*}
        &(x_2 - y_k + y_k - x_1)^{2m-2-r_k-s_k} \\ 
        =& \sum_{k_1+k_2=2m-2-r_k-s_k}\binom{2m-2-r_k-s_k}{k_1, k_2}(-1)^{k_1}(y_k - x_2)^{k_1}(y_k - x_1)^{k_2}.
    \end{align*}

    Each term in the expansion adds numbers into \(\boldpartitions{r'}{s'}\), among which only the term with \(k_1 = m-1-r_k\) and \(k_2 = m-1-s_k\) does not change \(J_{r'}\) and \(J_{s'}\). Other terms either adds 1 to \(j'_r\) or \(j'_s\). Now we expand factors \((x_2 - y_j + y_j - x_1)^{2m-1-r_j-s_j}\) for each \(j\notin J_r\cup J_s\cup \{k\}\), at which we then have either \(r_j+k_1\geq m, s_j+k_2\leq m-1\) or \(r_j+k_1\leq m-1, s_j+k_2\geq m\) for \(k_1 + k_2 = 2m-1-r_j-s_j\). Therefore, after all expansions at \(j\notin J_r\cup J_s\cup\{k\}\), \(\{1, \dots, q\}\backslash\{k\}\) is the disjoint union of updated \(J_{r'}\) and \(J_{s'}\) for each \(p_{[\mathbf{r}', \mathbf{s}']}(\underline{x}, \underline{y})(x_2 - x_1)^{e(\boldpartitions{r'}{s'})}\) in the expansion. 

    If \(j'_r > \frac{q-1}{2}\) or \(j'_s > \frac{q-1}{2}\), it falls back to \textbf{Case 2}. So the only possibly non-vanishing term must have \(j'_r = j'_s = \frac{q-1}{2}\). We consider for fixed \(J_{r'}, J_{s'}\), the factors \((y_k - x_2)^{k_1}(y_k - x_1)^{k_2}\) from the first expansion. When \((k_1, k_2) \neq (m-1-r_k, m-1-s_k)\), we either have \(r_k+k_1\geq m\) or \(s_k+k_2\geq m\), which adds 1 to either \(j'_r\) or \(j'_s\), and results in them vanishing. 

    Notice that we have an equality \(e(\boldpartitions{r'}{s'}) + \sum_{j=1}^q (r'_j + s'_j) = N\). In the above process, we subtract \(\sum_{j\notin J_r\cup J_s}(2m-1-r_j-s_j)-1\) from \(e(\boldpartitions{r'}{s'})\). Since we assumed \(r_j + s_j < 2m-1\) for some \(j\in J_r\cup J_s\) in \textbf{Case 3}, we have 

    \begin{align*}
        &e(\boldpartitions{r'}{s'}) - \sum_{j\notin J_r\cup J_s}(2m-1-r_j-s_j) + 1\\ 
        =&\sum_{j\in J_r}(r_j + s_j - m - s'_j) + \sum_{j\in J_s}(r_j + s_j - m - r'_j) + \sum_{j\notin J_r\cup J_s} (r_j + s_j) - (2m-1)(q-j_r - j_s)+1\\ 
        =&\sum_{j = 1}^{q}(r_j + s_j) - q(2m-1)+(m-1)(j_r + j_s) - \sum_{j\in J_r}s'_j - \sum_{j\in J_s}r'_j+1\\ 
        =& (m-1)(j_r + j_s) - \sum_{j\in J_r}s'_j - \sum_{j\in J_s}r'_j \geq 1.
    \end{align*}

    This means the only possibly non-vanishing term still has a factor \((x_2 - x_1)\). We then expand it again as \(x_2 - y_k + y_k - x_1\), which results in either \(j'_r=\frac{q+1}{2}\) or \(j'_s = \frac{q+1}{2}\). Thus, all terms in \textbf{Case 3} vanish. 
    
    Below is an illustration for this process in the case \(q = 3, m = 2\) for \(\boldpartitions{r}{s} = [2,1,1;0,3,1]\). 

    \[\begin{tikzcd}
        & \begin{array}{c} \begin{bmatrix}2,1,1\\ 0,3,1\end{bmatrix}\\ J_r = \{1\}, J_s = \{2\} \end{array} & \\
        \begin{array}{c} \begin{bmatrix}2,1,1\\ 0,2,2\end{bmatrix}\\ J_r = \{1\}, J_s = \{2,3\}\\ \text{vanish} \end{array} & \begin{array}{c} \begin{bmatrix}2,1,1\\ 0,2,1\end{bmatrix}(x_2 - x_1)\\ J_r = \{1\}, J_s = \{2\} \end{array} & \begin{array}{c} \begin{bmatrix}2,1,2\\ 0,3,1\end{bmatrix}\\ J_r = \{1,3\}, J_s = \{2\}\\ \text{vanish} \end{array}
        \arrow["\text{choose }k=3",from=1-2, to=2-2]
        \arrow[from=2-2, to=2-1]
        \arrow[from=2-2, to=2-3]
    \end{tikzcd}\]

    \textbf{Non-vanishing terms}: Now we exclude all three cases and arrive at \(\mathrm{nP}(N, 2q)\) in the proposition. We compute the coefficient of each term in the expansion of \(p_{\boldpartitions{r}{s}}(\underline{x}, \underline{y})\). At each \(j\notin J_r\cup J_s\cup \{k\}\), the expansion gives a summation 

    \[\sum_{k_{1j} = m-r_j}^{2m-1-r_j-s_j}(-1)^{1+r_j+s_j-k_{1j}}\binom{2m-1-r_j-s_j}{k_{1j}} \text{ or } \sum_{k_{2j}=m-s_j}^{2m-1-r_j-s_j}(-1)^{k_{2j}}\binom{2m-1-r_j-s_j}{k_{2j}}.\]

    With \zcref{Identity:alt-shifted-binom-sum}, they simplify to \((-1)^{m-1-s_j}\binom{2m-2-r_j-s_j}{m-1-r_j}\) and \((-1)^{m-1-r_j}\binom{2m-2-r_j-s_j}{m-1-s_j}\) respectively. We thus have 

    \[
    D_{[\mathbf{r}, \mathbf{s}]} = (-1)^{\frac{q-1}{2}-j_s}\binom{q-1-j_r-j_s}{\frac{q-1}{2}-j_s}\prod_{j\notin (J_r\cup J_s)}(-1)^{m-1-s_j}\binom{2m-2-r_j-s_j}{m-1-s_j}, \text{ for all } \boldpartitions{r}{s}\in \mathrm{nP}(N, 2q).
    \]

    Together with the multinomial coefficients, we have the desired form of \(D(q, m)\) in the proposition. 

\end{proof}

\begin{example}
    Let \(q = 1\). The moduli space \( \modulithetastable{\ktqm}{\underline{1}}\) is precisely \(\mathbb{P}^{m-1}\times \mathbb{P}^{m-1}\). Then we obtain the degree formula by plugging \(q = 1\) into \zcref{Proposition:k2qm-deg}. In the summation, we have \(j_r = j_s = 0, N = 2m-2 = r_1 + s_1\). The only possibility is \(r_1 = s_1 = m-1\). Hence, \(D(1, m) = \frac{(2m-2)!}{(m-1)!(m-1)!}\).
\end{example}

\begin{example}\label{Example:k2q1-deg}
    Let \(m = 1\). All binomial coefficients involving \(m\) and \(m-1\) evaluate to 1. For each \(j\in J_s\), we have \(r_j\leq m-1=0\), and \(\prod_{j\in J_s}(-1)^{r_j + s_j + 1} = (-1)^{\sum_{j\in J_s}(s_j + 1)} = (-1)^{\sum_{j\in J_s}s_j}(-1)^{j_s}\). Therefore, we obtain

    \[
    D(q, 1) = \sum_{[\mathbf{r}, \mathbf{s}]\in \mathrm{nP}(q-1, 2q)}\binom{q-1}{[\mathbf{r}, \mathbf{s}]} (-1)^{\frac{q-1}{2}}(-1)^{\sum_{j\in J_s}s_j}\binom{q-1-j_r-j_s}{\frac{q-1}{2}-j_s}.
    \]

    The value of each summand depends only on \(j_r\) and \(j_s\). The formula can be written as 

    \begin{align*}
        D(q, 1) &= \sum_{\substack{0\leq j_r\leq \frac{q-1}{2}\\0\leq j_s\leq \frac{q-1}{2}}}\binom{q}{j_r, j_s, q-j_r-j_s}\sum_{r + s = q-1}(-1)^{s + \frac{q-1}{2}}\binom{q-1}{r, s}\binom{q-1-j_r-j_s}{\frac{q-1}{2}-j_r}\\ 
        &\times\sum_{\substack{\sum_{j\in J_r}r_j=r\\ \sum_{j\in J_s}s_j=s}}\frac{r!}{\prod_{j\in J_r}r_j!}\frac{s!}{\prod_{j\in J_s}s_j!}.
    \end{align*}

    The last summation can be computed using inclusion-exclusion principle and written in terms of Stirling numbers of the second kind. 

    \[
    \sum_{\substack{\sum_{j\in J_r}r_j=r\\ \sum_{j\in J_s}s_j=s}}\frac{r!}{\prod_{j\in J_r}r_j!}\frac{s!}{\prod_{j\in J_s}s_j!} = \begin{Bmatrix}
        r\\j_r
    \end{Bmatrix}j_r!\begin{Bmatrix}
        s\\j_s
    \end{Bmatrix}j_s!
    \]

    Denote by \(E(j_r, j_s, q-1)\) the summation \(\sum_{r+s = q-1}(-1)^s\binom{q-1}{r, s}\begin{Bmatrix}
        r\\j_r
    \end{Bmatrix}j_r!\begin{Bmatrix}
        s\\j_s
    \end{Bmatrix}j_s!\). Then the degree formula becomes 

    \[
    D(q, 1) = \sum_{\substack{0\leq j_r\leq \frac{q-1}{2}\\0\leq j_s\leq \frac{q-1}{2}}}\binom{q}{j_r, j_s, q-j_r-j_s}\binom{q-1-j_r-j_s}{\frac{q-1}{2}-j_r}E(j_r, j_s, q-1).
    \]

    Some values of the degree are listed below.

        \begin{table}[H]
            \centering
            \begin{tabular}{|l|l|l|l|l|l|l|l|}
            \hline
            \(q\)    & 1 & 3 & 5   & 7     & 9       & 11  & 13        \\
            \hline
            \(D(q)\) & 1 & 6 & 230 & 23548 & 4675014 & 1527092468 & 743288515164\\ 
            \hline
            \end{tabular}
        \end{table}

\end{example}

We computed \(D(q, 1)\) for \(q\) up to 901 and discovered the following. 

\begin{conjecture}\label{Conjecture:deg-macmahon}
    The equality below holds for all odd \(q\geq 1\).

    \[
    D(q, 1) = \sum_{i = 0}^{\frac{q-1}{2}}(-1)^{\frac{q-1}{2}-i}\binom{q}{\frac{q-1}{2}}(2i + 1)^{q-1}
    \]
\end{conjecture}

\begin{remark}
    An AI--assisted proof is given in \zcref{Proposition:macmahon-conj}. We have verified the argument.
\end{remark}

\subsection{Identities in Chow groups of \texorpdfstring{\(\modulithetastable{\ktqm}{\underline{1}}\)}{M(K2qm, 1)}}\label{Subsection:4-2}

In this subsection, we study in detail relations in \(\mathrm{CH}^*(\modulithetastable{\ktqm}{\underline{1}})\) and work out some important identities for later use. 

\begin{proposition}\label{Proposition:k2qm-chow-eqs}
    For each \(i = 0, \dots, \frac{q-1}{2}\) and \(I\subset\{1, \dots, q\}\) with \(|I| = \frac{q+1}{2} + i\), we have the following equation in \(\mathrm{CH}^{(\frac{q+1}{2}+i)m}_{\mathbb{Q}}(\modulithetastable{\ktqm}{\underline{1}})\):

    \begin{align*}
        &m^{2i+1}\sum_{J\subset I}^{\frac{q-1}{2}-i}\prod_{j\in J}(y_j - x_1)^m\prod_{j\in I\backslash J}(y_j - x_1)^{m-1}(x_2 - x_1)^{2i+1} \\ 
        =& \sum_{\substack{l_j\leq m, j\in I\\ |\{j\in I~|~l_j=m\}|\leq \frac{q-1}{2}-i\\ |\{j\in I~|~l_j=m-1\}|\leq 2i+1}}(-1)^{\sum_{j\in I}l_j}\prod_{j\in I}\binom{m}{l_j}(y_j - x_1)^{l_j}(x_2 - x_1)^{(\frac{q+1}{2}+i)m-\sum_{j\in I}l_j}
    \end{align*}
\end{proposition}

\begin{proof}
    The case for \(i=0\) is just the expansion of \(T_I^{(3)}\). For a fixed \(i\in \{1, \dots, \frac{q-1}{2}\}\), we consider the following sum 

    \[
    S_{(i)}^I\coloneqq \sum_{k = 0}^i(-1)^k\binom{\frac{q-1}{2}-k}{i-k}\sum_{J\subset I}^k\prod_{j\in J}(y_j - x_1)^m\prod_{j\in I\backslash J}(y_j - x_2)^m.
    \]

    It is not hard to see that \(S_{(i)}^I\in I_{\text{taut}}\). 

    Without loss of generality, we always take \(I = \{1, \dots, \frac{q+1}{2} + i\}\). Notice that we can write \(y_j - x_2 = (y_j - x_1) - (x_2 - x_1)\). Therefore, after expanding powers of \((y_j - x_2)\) into powers of \((y_j - x_1)\) and \((x_2 - x_1)\), \(S_{(i)}^I\) can be written as a combination of terms of the form \(\prod_{j\in I}(y_j - x_1)^{l_j}(x_2 - x_1)^{e((l_j)_{j\in I})}\). There is no need to specify \(e((l_j)_{j\in I})\) since it always complements the degree so that the term lives in \(\mathrm{CH}^{(\frac{q+1}{2}+i)m}\). If \(\#\{j\in I~|~l_j=m\}\geq \frac{q+1}{2}\), then this term is in \(I_{\text{taut}}\). 
    
    Now consider a term \(\prod_{j\in J}(y_j - x_1)^m\prod_{j\in I\backslash J}(y_j - x_1)^{l_j}(x_2 - x_1)^e\), where \(|J|\in \{\frac{q+1}{2}-i, \dots, \frac{q-1}{2}\}\) and \(l_j < m, \text{ for all } j\in I\backslash J\). The coefficient of this term in \(S_{(i)}^I\) is 

    \[
    \prod_{j\in I\backslash J}\binom{m}{l_j}\sum_{k = 0}^i(-1)^k\binom{\frac{q-1}{2}-k}{i-k}\binom{\frac{q+1}{2}-j}{k}, 
    \]

    which equals 0 by \zcref{Identity:alt-binom-prod-sum}. As for a term with \(|J| < \frac{q+1}{2} - i\), the coefficient is \((-1)^{\sum_{j\in I}l_j}\prod_{j\in I}\binom{m}{l_j}\). Then the equation in the proposition is precisely \(0 = S_{(i)}^I\) after removing terms with coefficient 0 and moving the leading term to the left. 
\end{proof}

\subsection{Degrees of point configuration spaces}\label{Subsection:4-3}

Let \(q\geq 5\) be an odd integer, and \(S_q\) be the subspace quiver with vertices \(\{v_1, \dots, v_q, w\}\) and 1 arrow from \(v_i\) to \(w\) for all \(i = 1, \dots, q\). Consider the dimension vector \(\underline{d}\) that has 1 for each \(v_i\) and 2 for \(w\), which we will abbreviate as \((1^q;2)\) in the rest of this subsection. Then the moduli space of \(\tcan\)-stable representations parametrizes ordered \(q\)-point configurations on \(\mathbb{P}^1\) up to the \(PGL_2\)-action, where no more than \(\frac{q-1}{2}\) points coincide. The \(\tcan\)-strongly amply stable condition has been verified in \cite[Section 5.1]{frRe2021fano}. Moreover, the covering quiver \(\tilde{S}_q\) of \(S_q\) with respect to \(\underline{d}\) is precisely the opposite quiver of \(\ktqo\). By \zcref{Corollary:toric-red-degree}, the degree equation for \(X\coloneqq \modulithetastable{S_q}{\underline{d}}\) is equivalent to \((-K_{\tilde{X}})^{q-3}\Delta = D\bar{\rho}^{-1}([pt])\). In this case, the discriminant \(\Delta\) is simply \(x_2 - x_1\). 

Since \((\mathrm{CH}^{q-2}(\tilde{X}))^{\text{anti}}\simeq \mathrm{CH}^{q-3}(X)\) are both 1-dimensional over \(\mathbb{Q}\) and \(\Delta^{q-2}\) is anti-invariant under the \(W_{\underline{d}}\) action, it is immediate that we can write 

\begin{equation}
    (-K_{\tilde{X}})^{q-3}\Delta = q_1\Delta^{q-2}, \bar{\rho}^{-1}([pt]) = q_2\Delta^{q-2}
    \label{Equation:pt-config-deg-def}
\end{equation}

for some \(q_1, q_2\in \mathbb{Q}\). Then the degree is simply \(D = q_1 / q_2\). 

We start our computation with an alternative formula for \(\bar{\rho}^{-1}([pt])\). 

\begin{proposition}\label{Proposition:point-config-pt-toric}
    Let \(x_1, x_2, y_1, \dots, y_q\) be the generators of \(\mathrm{CH}^*(\tilde{X})\) as above. Then 

    \[
    \bar{\rho}^{-1}([pt]) = \frac{\Delta}{2^{q-3}}\sum_{i = 0}^{\frac{q-3}{2}}(i+1)\sum_J^{q-3-2i}\prod_{j\in J}(y_j - x_1 + y_j - x_2)\Delta^{2i}.
    \]
\end{proposition}

\begin{proof}
    We first use \zcref{Lemma:quiver-op-iso}. Writing \(X'\) for \(\modulithetastable{S_q^{op}}{\underline{d}}\), we have an isomorphism \(\mathrm{CH}^*(X)\simeq \mathrm{CH}^*(X')\), which then have the same generators \(u_1, u_2, z_1, \dots, z_q\). Therefore, we can compute \([pt]\) for \(X\) by applying \zcref{Equation:quiver-moduli-pt} to \(X'\),

    \[
    [pt] = \frac{\prod_{j = 1}^q(1 + z_q)}{(1 + u_1 + u_2)^2}\big|_{q-3}.
    \]

    The image of \([pt]\) under \(\bar{\rho}^{-1}\) is computed explicitly using \zcref{Equation:rho-bar-inv-def} and elementary symmetric functions \(e_1(y_i) = y_i, e_1(x_1, x_2) = x_1 + x_2, e_2(x_1, x_2) = x_1x_2\).
    \begin{align*}
        \bar{\rho}^{-1}([pt]) &= \frac{\Delta}{2^{q}}\frac{\prod_{j = 1}^q(2 + 2y_j)}{(1 + x_1)^2(1 + x_2)^2}\big|_{q-3}\\ 
        &= \frac{\Delta}{2^{q}}\frac{\prod_{j = 1}^q((y_j - x_1) + (y_j - x_2) + (1 + x_1) + (1 + x_2))}{(1 + x_1)^2(1 + x_2)^2}\big|_{q-3}\\ 
        &= \frac{\Delta}{2^{q}}\sum_{k = 0}^{q-3}\sum_J^{q-3-k}\prod_{j\in J}(y_j - x_1 + y_j - x_2)\left[\frac{(1 + x_1 + 1 + x_2)^{k+3}}{(1 + x_1)^2(1 + x_2)^2}\right]_{k}.
    \end{align*}

    We write \(H_k(x_1, x_2)\) for the component of homogeneous degree \(k\) of the power series expansion of \(\frac{(1 + x_1 + 1 + x_2)^{k+3}}{(1 + x_1)^2(1 + x_2)^2}\). It can be verified at odd and even degrees respectively that 

    \[
    H_k(x_1, x_2) = 
    \begin{cases}
        0, & \text{if \(k\) is odd;}\\ 
        4(k+2)\Delta^k, & \text{if \(k\) is even.}
    \end{cases}
    \]

    Substituting \(k = 2i\), we have the desired form of \([pt]\) in the proposition. 

\end{proof}

The next ingredient we need for degree computation is the following lemma.

\begin{lemma}\label{Lemma:point-config-alt-prod-sum-eq}
    For each \(0\leq j'\leq q-2\), the equation 

    \[
    \Delta^{q-2-j'}\sum_{J}^{j'}\left(\prod_{j\in J}(y_j - x_1) + (-1)^{j'}\prod_{j\in J}(y_j - x_2)\right) = D_{j'}\Delta^{q-2}
    \]
    
    holds in \(\mathrm{CH}^{q-2}(\tilde{X})\), where \(D_{j'} = \frac{2q}{q-j'}\binom{\frac{q-1}{2}}{j'}\).
\end{lemma}

\begin{proof}

    The polynomial on the left is an anti-invariant in \(\mathrm{CH}^{q-2}(\tilde{X})\), hence equals \(D_{j'}\Delta^{q-2}\) for some \(D_{j'}\in \mathbb{Q}\). This is true because \(\mathrm{CH}^{q-2}(\tilde{X})^{\text{anti}}\simeq \mathrm{CH}^{q-3}(X)\) is a one-dimensional \(\mathbb{Q}\)-vector space. 
    To compute these \(D_{j'}\), we plug \(m=1\) into the equation in \zcref{Proposition:k2qm-chow-eqs}. It gives rise to a relation in \(\mathrm{CH}^{\frac{q+1}{2} + i}(\tilde{X})\)

    \[
    \sum_{k = 0}^{\frac{q-1}{2}-i}\binom{\frac{q-1}{2}-k}{i}(-\Delta)^{\frac{q+1}{2}+i - k}\sum_{J\subset I}^k\prod_{j\in J}(y_j - x_1) = 0
    \]

    for each \(0\leq i\leq \frac{q-5}{2}\) and \(I\subset \{1, \dots, q\}\) with \(|I| = \frac{q+1}{2} + i\).

    The same procedure works when we switch \(x_1\) and \(x_2\) in the process and the result is 

    \[
    \sum_{k = 0}^{\frac{q-1}{2}-i}\binom{\frac{q-1}{2}-k}{i}\Delta^{\frac{q+1}{2}+i - k}\sum_{\substack{K\subset I \\ |K| = k}}\prod_{j\in K}(y_j - x_2) = 0.
    \]

    We multiply both equations by \(\Delta^{\frac{q-5}{2}-i}\) to make sure they live in \(\mathrm{CH}^{q-2}(\tilde{X})^{\text{anti}}\). Adding them up gives an equation for each \(i\) concerning \(D_0, \dots, D_{\frac{q-1}{2}-i}\):

    \begin{align*}
        &\sum_{k = 0}^{\frac{q-1}{2}-i}(-1)^k\frac{\binom{\frac{q-1}{2}-k}{i}\binom{q}{\frac{q+1}{2}+i}\binom{\frac{q+1}{2}+i}{k}}{\binom{q}{k}}\Delta^{q-2-k}\sum_{K}^{k}\left(\prod_{j\in K}(y_j - x_1) + (-1)^k\prod_{j\in K}(y_j - x_2)\right)\\ 
        =&\sum_{k = 0}^{\frac{q-1}{2}-i}(-1)^k\binom{q-k}{\frac{q-1}{2}-i}\binom{\frac{q-1}{2}-k}{i}D_k\Delta^{q-2} = 0.
    \end{align*}

    We can then compute \(D_{j'}\) via recursion

    \[
    D_{j'} = (-1)^{j+1}\frac{\sum\limits_{k = 0}^{j' - 1}(-1)^k\binom{q-k}{j'}\binom{\frac{q-1}{2}-k}{\frac{q-1}{2}-j'}}{\binom{q-j'}{j'}}D_k.
    \]

    The base cases are \(D_0 = 2, D_1 = q\). Assume now \(D_{l} = \frac{2q}{q-l}\binom{\frac{q-1}{2}}{l}\) holds for \(l = 0, \dots, j'-1\). Then 

    \begin{align*}
        D_{j'} &= (-1)^{j'+1}\frac{\sum\limits_{k = 0}^{j' - 1}(-1)^k\binom{q-k}{j'}\binom{\frac{q-1}{2}-k}{\frac{q-1}{2}-j'}}{\binom{q-j'}{j'}}\frac{2q}{q-k}\binom{\frac{q-1}{2}}{k}\\  
        &=(-1)^{j'+1}\frac{2q\binom{\frac{q-1}{2}}{j'}}{\binom{q-j'}{j'}}\sum_{k = 0}^{j' - 1}(-1)^k\frac{1}{j'}\binom{j'}{k}\binom{q-1-k}{j'-1},
    \end{align*}

    and the sum equals \((-1)^{j'+1}\binom{j'}{j'}\binom{q-1-j'}{j'-1}\) due to \zcref{Identity:shifted-binom-diff}. Thus, we conclude that 
    
    \[
    D_{j'} = (-1)^{j'+1}\frac{2q}{\binom{q-j'}{j'}}\binom{\frac{q-1}{2}}{j'}(-1)^{j'+1}\frac{\binom{q-1-j'}{j'-1}}{j'} = \frac{2q}{q-j'}\binom{\frac{q-1}{2}}{j'}
    \]

\end{proof}

The next step is to compute \(q_1\) and \(q_2\) in \zcref{Equation:pt-config-deg-def}. 

\begin{align*}
    (-K_{\tilde{X}})^{q-3}\Delta &= \Delta\sum_{[\mathbf{r}, \mathbf{s}]\in \mathrm{P}(q-3,2q)}\binom{q-3}{[\mathbf{r}, \mathbf{s}]}\prod_{j = 1}^q(y_j - x_1)^{r_j}(y_j - x_2)^{s_j} = q_1\Delta^{q-2}\\ 
    \bar{\rho}^{-1}([pt]) &= \frac{\Delta}{2^{q-3}}\sum_{i = 0}^{\frac{q-3}{2}}(i+1)\sum_J^{q-3-2i}\prod_{j\in J}(y_j - x_1 + y_j - x_2)\Delta^{2i} = q_2\Delta^{q-2}
\end{align*}

The strategy is similar to the proof of \zcref{Proposition:k2qm-deg}. Every term in \((-K_{\tilde{X}})^{q-3}\Delta\) can be written as \(\Delta\prod_{j\in J_r}(y_j - x_1)^{r_j}\prod_{j\in J_s}(y_2 - x_2)^{s_j}\) such that \(\sum_{j\in J_r}r_j + \sum_{j\in J_s}s_j = q-3\). We first have the following reduction via \(T_j^{(1)}\). 

\[
\prod_{j\in J_r}(y_j - x_1)^{r_j}\prod_{j\in J_s}(y_j - x_2)^{s_j}=(-1)^{\sum_{j\in J_s} (s_j - 1)}\prod_{j\in J_r}(y_j - x_1)\prod_{j\in J_s}(y_j - x_2)\Delta^{q-3-j_r-j_s}.
\]

This shows that terms with the same \(J_r\) and \(J_s\) only differs by a sign. In fact, their value only depends on \(j_r, j_s\). We can then reindex the sum in \((-K_{\tilde{X}})^{q-3}\Delta\) by \(j_r\) and \(j_s\). The coefficient of the term \(\prod_{j\in J_r}(y_j - x_1)\prod_{j\in J_s}(y_j - x_2)\Delta^{q-2-j_r-j_s}\) after reindexing is 

\begin{align*}
    &\sum_{\substack{\sum_{j\in J_r}r_j + \sum_{j\in J_s}s_j=q-3\\ r_j\geq 1, \text{ for all } j\in J_r, s_j\geq 1, \text{ for all } j\in J_s}}\frac{(q-3)!}{\prod_{j\in J_r}r_j!\prod_{j\in J_s}s_j!}(-1)^{\sum_{j\in J_s}(s_j - 1)} \\ 
    =& (-1)^{j_s}\sum_{p_1 + p_2 = q-3}(-1)^{p_2}\binom{q-3}{p_1}\begin{Bmatrix}
        p_1\\ 
        j_r
    \end{Bmatrix}j_r!\begin{Bmatrix}
        p_2\\ 
        j_s
    \end{Bmatrix}j_s!,
\end{align*}

which is precisely \((-1)^{j_s}E(j_r, j_s, q-3)\), a sum we introduced in \zcref{Example:k2q1-deg}. Notice a clear duality here \(E(j_r, j_s, q-3) = E(j_s, j_r, q-3)\). So we want to compute the sum \(((-1)^{j_r}\prod_{j\in J_r}(y_j - x_1)\prod_{j\in J_s}(y_j - x_2) + (-1)^{j_s}\prod_{j\in J_r}(y_j - x_2)\prod_{j\in J_s}(y_j - x_1))\Delta^{q-2-j_r-j_s}\) for a pair \((j_r, j_s)\) where \(j_r\geq j_s\). 

Again, we apply the substitution \(y_j - x_2 = y_j - x_1 - \Delta\) and \(y_j - x_1 = y_j - x_2 + \Delta\) to the above sum and obtain 

\[
\sum_{k = 0}^{j_s}\sum_{\substack{K\subset J_s\\ |K|=k}}\left[(-1)^k\prod_{j\in J_r\cup J_s\setminus K}(y_j - x_1) + (-1)^{q-3-j_r-j_s}\prod_{j\in J_r\cup J_s\setminus K}(y_j - x_2)\right]\Delta^k.
\]

Now we sum it over all subsets of \(\{1, \dots, q\}\) of same sizes, which allows us to apply \zcref{Lemma:point-config-alt-prod-sum-eq} to convert these sums into multiples of \(\Delta^{q-2}\).

\begin{align*}
    &(-K_{\tilde{X}})^{q-3}\Delta \\ 
    =&\left[\left(\sum_{j' = 1}^{q-3}\sum_{j_s = 0}^{\lceil \frac{j}{2} \rceil - 1}\right) + \left(\frac{1}{2}\sum_{j_r = 0, j_s = j_r}^{\frac{q-3}{2}}\right)\right] (-1)^{j_s}E(j_r, j_s, q-3)\sum_{k = 0}^{j_s}(-1)^k\frac{\binom{q}{j_r}\binom{q-j_r}{j_s}\binom{j_s}{k}}{\binom{q}{j_r + j_s - k}}D_{j'-k}\Delta^{q-2}\\ 
    =& q_1\Delta^{q-2}.
\end{align*}

The same procedure applies to every term in \(\bar{\rho}^{-1}([pt])\), and we thus obtain 

\begin{align*}
    \bar{\rho}^{-1}([pt]) &= \frac{\Delta^{q-2}}{2^{q-2}}(q-1) + \frac{\Delta^{q-2}}{2^{q - 3}}\sum_{i = 0}^{\frac{q-3}{2}}(i+1)\\ 
    &\times\left[\left(\sum_{j'=0}^{\frac{q-5}{2}-i}\sum_{k=0}^j\right) + \frac{1}{2}\left(\sum_{k=0}^{\frac{q-3}{2}-i}\right)\right](-1)^k\binom{2i+3+k}{k}\binom{q-3-2i-k}{j'-k}D_{q-2i-3-k}\\ 
    &= q_2\Delta^{q-2}.
\end{align*}

By dividing \(q_1\Delta^{q-2}\) by \(q_2\Delta^{q-2}\) we obtain the degree formula.

\begin{proposition}\label{Proposition:point-config-deg}
    The degree of the moduli space of stable \(q\)-point configurations on \(\mathbb{P}^1\) is 

    \begin{equation*}
        D(q) = \frac{2^{q-3}\left(\sum\limits_{j'=1}^{q-3}\sum\limits_{j_s=0}^{\lceil \frac{j'}{2}\rceil-1}+\frac{1}{2}\sum\limits_{j_r=0,j_s=j_r}^{\frac{q-3}{2}}\right)(-1)^{j_s}E(j_r, j_s, q-3)\sum\limits_{k=0}^{j_s}(-1)^k\binom{j'-k}{j_r}\binom{q+k-j'}{k}D_{j'-k}}{\frac{q-1}{2} + \sum\limits_{i = 0}^{\frac{q-5}{2}}(i+1)\left(\sum\limits_{j' = 0}^{\frac{q-5}{2}-i}\sum\limits_{k=0}^{j'} + \frac{1}{2}\sum\limits_{\substack{k=0\\j'=\frac{q-3}{2}-i}}^{\frac{q-3}{2}-i}\right)(-1)^k\binom{2i+3+k}{k}\binom{q-2i-3-k}{j'-k}D_{q-2i-3-k}},
    \end{equation*}

    where \(D_{j'} = 2q\binom{\frac{q-1}{2}}{j'} / (q-j')\) for \(0\leq j'\leq q-2\).
\end{proposition}

\begin{example}
    Let \(q = 5\). It has been established in \cite[Section 5.1]{frRe2021fano} that \(M^{\tcan\text{-st}}(S_5, (1^5;2))\) is isomorphic to the blowup of \(\mathbb{P}^2\) in 4 general points. The coefficients \(D_{j'}\) are

    \[
    D_0 = 2, D_1 = 5, D_2 =10/3.
    \]

    The degree is 

    \[
    D(5) = 4\frac{D_1 + 2D_2 + 2D_2 - 4D_1}{2 + D_2 + D_2 + 2D_1} = 5.
    \]

    We see that the degree indeed agrees with that in \cite{fanography}. 
\end{example}

\begin{example}
    Let \(q = 7\). Then \(M^{\tcan\text{-st}}(S_7, (1^7;2))\) is a Fano 4-fold of Picard rank 7, index 1, and the numbers \(D_{j'}\) are as below. 

    \[
    D_0 = 2, D_1 = 7, D_2 = \frac{42}{5}, D_3 = \frac{7}{2}.
    \]

    Then the degree is 

    \begin{align*}
        D(7) &= 16\frac{D_1 + 16D_2 - 6D_1 + 36D_3 + 36D_3 - 60D_2 + 24D_4 + 96D_4 - 96D_3 + 72D_4 - 144D_3 + 120D_2}{3 + D_4 + 4D_4 - 4D_3 + 3D_4 - 6D_3 + 5D_2 + D_2 + 2D_2 - 6D_1}\\ 
        &= 154.
    \end{align*}
\end{example}

We have verified these degrees using the \cite{quivertools} package for \(q = 5, 7, \dots, 13\). See \zcref{appendix:B} for more details. 

\newpage

\appendix

\section{Combinatorial identities}\label{appendix:A}

\begin{identity}[Alternating shifted binomial sum]
    \label{Identity:alt-shifted-binom-sum}
    For \(0\leq i\leq l\), we have

    \[
    \sum_{k = 0}^i(-1)^{i - k}\binom{l + i}{l + k} = (-1)^i\binom{l + i - 1}{i}
    \]
\end{identity}

\begin{proof}
    We apply the Pascal's identity 

    \[
    \binom{n}{m} = \binom{n - 1}{m} + \binom{n - 1}{m - 1}
    \]

    to each summand \(\binom{l + i}{l + k}\) with the substitution \(n = l + i, m = l + k\) to obtain

    \begin{align*}
        \sum_{k = 0}^i(-1)^{i - k}\binom{l + i}{l + k} &= \sum_{k = 0}^i(-1)^{i - k}\left(\binom{l + i - 1}{l + k} + \binom{l + i - 1}{l + k - 1}\right)\\ 
        &= (-1)^i\binom{l + i - 1}{l - 1} + \binom{l + i - 1}{l + i}
    \end{align*}

    By convention, a binomial coefficient \(\binom{n}{k}\) with \(n < k\) is 0. Thus, we obtained the desired identity. 

\end{proof}

\begin{identity}[Alternating binomial product]
    \label{Identity:alt-binom-prod-sum}
    For any \(0 < j\leq i\leq l\),

    \[
    \sum_{k = 0}^i(-1)^k\binom{l-k}{i-k}\binom{l+1-j}{k} = 0.
    \]
    
\end{identity}

\begin{proof}

    The sum is the coefficient of \(t^{l-i}\) in \(\sum\limits_{k = 0}^{l}(-1)^k\binom{l+1-j}{k}(1+t)^{l-k}\), which equals
    
    \begin{align*}
     &\sum_{k = 0}^l(-1)^k\binom{l+1-j}{k}(1+t)^{l-k}\\
     =& (1+t)^{l}\sum_{k = 0}^l\binom{l+1-j}{k}(-\frac{1}{1+t})^k\\ 
     =& (1+t)^{l}(1-\frac{1}{1+t})^{l-j}\\ 
     =& (1+t)^{j-1}t^{l+1-j}. 
    \end{align*}
    
    Since $j\leq i$, we always have $l+1-j > l-i$. Hence, the coefficient of $t^{l-i}$ is 0. 
\end{proof}

\begin{identity}[Shifted binomial difference]\label{Identity:shifted-binom-diff}
    Let \(t, s, l\) be non-negative integers. Then 
    \[
    \sum_{k = 0}^l(-1)^k\binom{l}{k}\binom{t - k}{s} = \binom{t - l}{s - l}.
    \]
\end{identity}

\begin{proof}

    Notice that \(\binom{t - k}{s}\) is the coefficient of \(x^s\) in the expansion of \((1 + x)^{t - k}\). Therefore, the left-hand side of the identity equals the coefficient of \(x^s\) in the expansion of 
    
    \begin{align*}
        \sum_{k = 0}^{l}(-1)^k\binom{l}{k}(1 + x)^{t - k}=&(1 + x)^t\sum_{k = 0}^l\binom{l}{k}(-\frac{1}{1 + x})^k\\ 
        =&(1 + x)^t(1 - \frac{1}{1+x})^l\\ 
        =&x^l(1 + x)^{t - l},
    \end{align*}
    
    which is \(\binom{t - l}{s - l}\). Thus, the identity holds. 
    
\end{proof}

\section{SageMath Codes}\label{appendix:B}

In this section, we append source codes for degree formulas in \zcref{Proposition:k2qm-deg} and \zcref{Proposition:point-config-deg}. 

\subsection{SageMath codes for \texorpdfstring{\(D(q, m)\)}{D(q,m)}}

\begin{lstlisting}[
  language=Python,
  breaklines=true,
  breakatwhitespace=false
]
from sage.all import *
from itertools import product, combinations, permutations
from collections import Counter
import numbers
import numpy as np

# determine whether a term vanishes
def vanishing_term(expo_list_r, expo_list_s, J_r, J_s, q, m):
    # case 1
    if (set(J_r) & set(J_s)) != set([]):
        return True
    # case 2
    elif (len(J_r) >= (q+1)/2) | (len(J_s) >= (q+1)/2) :
        return True
    # case 3
    elif any([expo_list_r[i] + expo_list_s[i] < 2*m-1 for i in J_r] + [expo_list_r[i] + expo_list_s[i] < 2*m-1 for i in J_s]):
        return True
    else:
        return False

# compute the degree
def deg_k2qm(q, m):
    D = 0
    l = (q-1)/2
    for int_vec in IntegerVectors(2*q*m-q-1, 2*q):
        expo_list_r = int_vec[0:q]
        expo_list_s = int_vec[q:]
        J_r = np.where(np.array(expo_list_r) >= m)[0]
        J_s = np.where(np.array(expo_list_s) >= m)[0]
        j_r = len(J_r)
        j_s = len(J_s)
        
        if vanishing_term(expo_list_r, expo_list_s, J_r, J_s, q, m) == False:
            temp_sign = (-1)**(l-j_s) * (-1)**(sum([m-1-expo_list_s[j] for j in range(q) if j not in list(J_r)+list(J_s)])) * (-1)**(sum([expo_list_r[j]+expo_list_s[j]-2*m+1 for j in J_s]))
            temp_D = temp_sign * binomial(q-1-j_r-j_s, l-j_r) * prod([binomial(2*m-2-expo_list_r[j]-expo_list_s[j], m-1-expo_list_s[j]) for j in range(q) if j not in list(J_r)+list(J_s)])
            temp_D *= multinomial(list(int_vec)) * prod([binomial(expo_list_r[j]-m, m-1-expo_list_s[j]) for j in J_r]) * prod([binomial(expo_list_s[j]-m, m-1-expo_list_r[j]) for j in J_s])
            D += temp_D
    
    return D * m**(2*q*m - q - 1)
\end{lstlisting}

We compute the degrees in the case \((q, m) = (3, 1), (3, 2), (3, 3)\) and compare them with degrees obtained by using \cite{quivertools}. 

\begin{lstlisting}[
  language=Python,
  breaklines=true,
  breakatwhitespace=false
]
sage: from degree_formula_k2qm import *
sage: [deg_k2qm(3,i) for i in [1,2,3]]
[6, 12128256, 4678810092768240]
sage: from quiver import *
sage: for m in range(1, 4):
....:     M = matrix(5, 5)
....:     for i in range(2,5):
....:         M[0,i] = m
....:         M[1,i] = m
....:     Q = Quiver.from_matrix(M)
....:     X = QuiverModuliSpace(Q, (1,1,1,1,1))
....:     print(X.degree(Q.canonical_stability_parameter((1,1,1,1,1))))
....: 
6
12128256
4678810092768240
\end{lstlisting}

\subsection{SageMath codes for \texorpdfstring{\(D(q)\)}{D(q)}}

SageMath codes for degrees of \(q\)-point configuration spaces on \(\mathbb{P}^1\):

\begin{lstlisting}[
  language=Python,
  breaklines=true,
  breakatwhitespace=false
]
from sage.all import *
import math

def stirling_sum(jr,js,t):
    return sum([(-1)^(p2) * binomial(t, p2) * stirling_number2(t-p2, jr) * factorial(jr) * stirling_number2(p2,js) * factorial(js) for p2 in range(js, t+1-jr)])

def pt_config_deg(q):

    # compute the coefficients D_{j'}
    D = [2*q*binomial((q-1)/2, l) / (q-l) for l in range(q-2)]

    # compute the q_1 for the anti-canonical class
    AKX = 0
    for j in range(1, q-2):
        for js in range(floor(j/2)+1):
            jr = j - js
            temp_factor = (-1)^(js) * stirling_sum(jr,js,q-3)
            temp_sum = 0
            for k in range(js+1):
                temp_sum += (-1)^k * binomial(j-k, jr) * binomial(q+k-j, k) * D[j-k]
                print(jr,js,j - k, (-1)^k * binomial(j-k, jr) * binomial(q+k-j, k), temp_factor)
            
            if jr == js:
                AKX += temp_factor * temp_sum / 2
            else:
                AKX += temp_factor * temp_sum

    # compute q_2 for the point class
    PT = 0
    for i in range((q-3)/2):
        temp_factor = i+1
        
        for j in range((q-1)/2-i):
            temp_sum = 0
            for k in range(j+1):
                temp_sum += (-1)^k * binomial(2*i+3+k,k) * binomial(q-3-2*i-k, j-k) * D[q-3-2*i-k]
            if 2*j == q-3-2*i:
                PT += temp_factor * temp_sum / 2
            else:
                PT += temp_factor * temp_sum

    PT += (q-1)/2
    PT /= 2^(q-3)
    return AKX / PT
\end{lstlisting}

We compute the degrees for \(q = 5, 7, \dots, 13\) and compare them with degrees obtained by using \cite{quivertools}. 

\begin{lstlisting}[
  language=Python,
  breaklines=true,
  breakatwhitespace=false
]
sage: from degree_formula_Sq import *
sage: [pt_config_deg(q) for q in [5,7,9,11,13]]
[5, 154, 13005, 2189726, 620169186]
sage: from quiver import *
sage: for q in range(5,15,2):
....:     Sq = SubspaceQuiver(q)
....:     d = tuple([1]*q+[2])
....:     Xq = QuiverModuliSpace(Sq, d)
....:     print(Xq.degree(Sq.canonical_stability_parameter(d)))
....: 
5
154
13005
2189726
620169186
\end{lstlisting}

\section{Proof of Conjecture 1}\label{appendix:C}

\begin{proposition}
    \label{Proposition:macmahon-conj}
    Let \(q\geq 1\) be an odd integer. Then we have
    \[
    D(q, 1) = \sum_{i = 0}^{\frac{q-1}{2}}(-1)^{\frac{q-1}{2}-i}\binom{q}{\frac{q-1}{2}-i}(2i+1)^{q-1}
    \]
\end{proposition}

\begin{proof}

Writing \(n = \frac{q-1}{2}\), we shall prove that \(D(q, 1)\) is equal to the central MacMahon numbers \(a(n)\coloneqq \sum_{i = 0}^n(-1)^{n-i}\binom{2n+1}{n-i}(2i+1)^{2n}\)(OEIS entry A177043). 

We first introduce some notations. Let \(\binom{x}{j}\) be the polynomial \(\frac{1}{j!}x(x-1)\cdots (x-j+1)\) in a variable \(x\) for \(j\geq 0\). Recall from \cite[Section 6.1]{graham1989concrete} that the generating function for Stirling numbers of the second kind is 

\[
x^k = \sum_{j = 0}^k \begin{Bmatrix}
    k\\ j
\end{Bmatrix}j!\binom{x}{j}, k\geq 0.
\]

Note that the original generating function is written as 

\[
x^k = \sum_{j = 0}^k\begin{Bmatrix}
    k\\ j
\end{Bmatrix}x^{\underline{j}},
\]

where \(x^{\underline{j}}\coloneqq x(x - 1)\cdots (x - j + 1)\). By cancelling the factor \(j!\), we see that the two functions are the same. 

We expand \((x - y)^{2n}\) as follows 

\begin{align*}
    (x - y)^{2n} &= \sum_{r + s = 2n}(-1)^s\binom{2n}{s}x^r y^s\\ 
    &= \sum_{r + s = 2n}(-1)^s\binom{2n}{s}\sum_{j_r = 0}^r\begin{Bmatrix}
        r\\ j_r
    \end{Bmatrix}j_r!\binom{x}{j_r}\sum_{j_s = 0}^s\begin{Bmatrix}
        s\\ j_s
    \end{Bmatrix}j_s!\binom{y}{j_s}.
\end{align*}

For the convenience of changing order of sums, we extend the latter two sums to \(\sum_{j_r = 0}^{2n}\) and \(\sum_{j_s = 0}^{2n}\), which can be done since \(\begin{Bmatrix}
    r\\ j_r
\end{Bmatrix} = \begin{Bmatrix}
    s\\ j_s
\end{Bmatrix} = 0\) for \(j_r > r\) and \(j_s > s\). This gives us 

\begin{align*}
    (x - y)^{2n} &= \sum_{j_r = 0}^{2n}\sum_{j_s = 0}^{2n}\binom{x}{j_r}\binom{y}{j_s}\sum_{r + s = 2n}(-1)^s\binom{2n}{s}\begin{Bmatrix}
        r\\ j_r
    \end{Bmatrix}j_r!\begin{Bmatrix}
        s\\ j_s
    \end{Bmatrix}j_s!\\ 
    &= \sum_{j_r = 0}^{2n}\sum_{j_s = 0}^{2n}\binom{x}{j_r}\binom{y}{j_s}E(j_r, j_s, 2n).
\end{align*}

Substituting \(i\mapsto n-i\), we have 

\[
a(n) = \sum_{i = 0}^n(-1)^i\binom{2n+1}{i}(2n+1-2i)^{2n}. 
\]

Let \(x = 2n+1-i, y = i\). Then we apply \((x - y)^{2n}\) to \(a(n)\) and obtain 

\begin{align*}
    a(n) &= \sum_{i = 0}^n(-1)^i\binom{2n+1}{i}\sum_{j_r = 0}^{2n}\sum_{j_s = 0}^{2n}\binom{2n+1-i}{j_r}\binom{i}{j_s}E(j_r, j_s, 2n)\\ 
    &= \sum_{j_r = 0}^{2n}\sum_{j_s = 0}^{2n}E(j_r, j_s, 2n)\sum_{i = 0}^n(-1)^i\binom{2n+1}{i}\binom{2n+1-i}{j_r}\binom{i}{j_s}. 
\end{align*}

We analyze the index sets in detail. 

If \(j_r \geq n + 1\) and \(j_s \geq n + 1\), each summand in \(E(j_r, j_s, 2n)\) is 0 because any pair \((r, s)\) with \(r + s = 2n\) would result in \(r < j_r\) or \(s < j_s\), which means \(\begin{Bmatrix}
    r\\ j_r
\end{Bmatrix} = 0\) (resp. \(\begin{Bmatrix}
    s\\ j_s
\end{Bmatrix} = 0\)).

If \(j_s\leq n\) and \(j_r\geq n + 1\), we discuss four cases. 

\begin{itemize}
    \item \(j_s = n, j_r > n + 1\): the last sum is \((-1)^n\binom{2n+1}{n}\binom{n+1}{j_r}\), which is 0 since \(\binom{n+1}{j_r} = 0\) for \(j_r > n + 1\);
    \item \(j_s = n, j_r = n + 1\): \(E(j_r, j_s, 2n) = 0\) for the same reason as in the case \(j_r \geq n + 1, j_s \geq n + 1\);
    \item \(j_s < n, j_r\geq 2n + 1 - j_s\): the only possibly non-zero terms in \(E(j_r, j_s, 2n)\) are those with \(s\geq j_s\), which implies \(2n-r\geq j_s\), hence \(r\leq 2n-j_s < 2n + 1 - j_s \leq j_r\). Now that \(r < j_r\), these possibly non-zero terms also vanish because \(\begin{Bmatrix}
        r\\ j_r
    \end{Bmatrix} = 0\).
    \item \(j_s < n, n + 1\leq j_r\leq 2n - j_s\): we compute the last sum as follows: 
    
    \begin{align*}
        &\sum_{i = 0}^n(-1)^i\binom{2n+1}{i}\binom{2n+1-i}{j_r}\binom{i}{j_s}\\ 
        =&\sum_{i = j_s}^{2n + 1 - j_r}(-1)^i\binom{2n+1}{i}\binom{2n+1-i}{j_r}\binom{i}{j_s}\\ 
        =&\sum_{i = j_s}^{2n + 1 - j_r}(-1)^{j_s + i - j_s}\frac{(2n+1)!}{i!(2n+1-i)!}\frac{(2n+1-i)!}{j_r!(2n+1-i-j_r)!}\frac{i!}{j_s!(i-j_s)!}\\ 
        =&\sum_{l = 0}^{2n + 1 - j_r - j_s}(-1)^{j_s}(-1)^l\frac{(2n+1)!}{j_r!j_s!(2n+1-j_r-j_s)!}\frac{(2n+1-j_r-j_s)!}{(i-j_s)!(2n+1-i-j_r)!}\\ 
        =&(-1)^{j_s}\sum_{l = 0}^{2n + 1 - j_r - j_s}(-1)^l\binom{2n+1}{j_r, j_s, 2n+1-j_r-j_s}\binom{2n+1-j_r-j_s}{l}\\ 
        =&(-1)^{j_s}\binom{2n + 1}{j_r, j_s, 2n + 1 - j_r - j_s}\sum_{l = 0}^{2n + 1 - j_r - j_s}(-1)^l\binom{2n + 1 - j_r - j_s}{l}\\ 
        =&(-1)^{j_s}\binom{2n + 1}{j_r, j_s, 2n + 1 - j_r - j_s}(1 - 1)^{2n + 1 - j_r - j_s} = 0.
    \end{align*}

    The reason for the first equality is, \(\binom{i}{j_s} = 0, \text{ for all } i < j_s\) and \(\binom{2n + 1 - i}{j_r} = 0, \text{ for all } i > 2n + 1 - j_r\). Since \(j_r\geq n + 1\), \(2n + 1 - j_r\leq n\) will always hold. So the sum \(\sum_{i = 0}^n\) can be simplified to \(\sum_{i = j_s}^{2n + 1 - j_r}\). 

    The second equality can be expanded as above. Then the last equality is the binomial theorem \((a + b)^k = \sum_{i=0}^k\binom{k}{i}a^{n-i}b^i\) with \(a=1, b=-1, k=2n+1-j_r-j_s\). 
\end{itemize}

Then the dual case \(j_r\leq n, j_s\geq n + 1\) is analogous to the above discussion. Therefore, we are left with the following sum. 

\[
a(n) = \sum_{j_r = 0}^{n}\sum_{j_s = 0}^{n}E(j_r, j_s, 2n)\sum_{i = 0}^n(-1)^i\binom{2n+1}{i}\binom{2n+1-i}{j_r}\binom{i}{j_s}.
\]

For any \(0\leq j_r\leq n, 0\leq j_s\leq n\), and \(0\leq i\leq n\), the summand \(\binom{2n+1}{i}\binom{2n+1-i}{j_r}\binom{i}{j_s}\) does not vanish. Thus, the last sum is computed as 

\begin{align*}
    &\sum_{i = 0}^n(-1)^i\binom{2n+1}{i}\binom{2n+1-i}{j_r}\binom{i}{j_s}\\ 
    =&\binom{2n + 1}{j_r, j_s, 2n + 1 - j_r - j_s}\sum_{i = j_s}^n(-1)^i\binom{2n + 1 - j_r - j_s}{i - j_s}\\ 
    =&\binom{2n + 1}{j_r, j_s, 2n + 1 - j_r - j_s}(-1)^n\binom{2n - j_r - j_s}{n - j_r}.
\end{align*}

The first equality uses the same expansion as in the fourth bullet point in the above discussion, and the second equality can be obtained by the following successive application of the Pascal's identity. 

\begin{align*}
    &\sum_{i = j_s}^n(-1)^i\binom{2n + 1 - j_r - j_s}{i - j_s}\\
    =&(-1)^{j_s}\sum_{l = 0}^{n - j_s}(-1)^l\binom{2n + 1 - j_r - j_s}{l}\\ 
    =&(-1)^{j_s}\sum_{l = 0}^{n - j_s}(-1)^l\left[\binom{2n - j_r - j_s}{l} + \binom{2n - j_r - j_s}{l-1}\right]\\ 
    =&(-1)^{j_s}\left[\binom{2n - j_r - j_s}{-1} + \binom{2n - j_r - j_s}{0} - \binom{2n - j_r - j_s}{0} - \binom{2n - j_r - j_s}{1}\right. \\ 
    &\left.+ \cdots + (-1)^{n-j_s}\binom{2n - j_r - j_s}{n - j_s - 1} + (-1)^{n - j_s}\binom{2n - j_r - j_s}{n - j_s}\right]\\ 
    =&(-1)^{j_s}\left[\binom{2n - j_r - j_s}{-1} + (-1)^{n - j_s}\binom{2n - j_r - j_s}{n - j_s}\right]\\ 
    =&(-1)^n\binom{2n - j_r - j_s}{n - j_s}.
\end{align*}

The last equality is due to the fact that \(\binom{2n - j_r - j_s}{-1} = 0, \text{ for all } 2n - j_r - j_s\geq 0\). For a more rigorous treatment of binomial coefficients with negative arguments, we refer to \cite{graham1989concrete}. 

Plugging it back into \(a(n)\), we get precisely \(D(q, 1)\). 
\end{proof}

\section{AI and computation resource disclosure}\label{appendix:D}

Computations of invariants of quiver moduli in this paper are verified using \cite{quivertools}. The proof of \zcref{Proposition:macmahon-conj} is achieved with the help of a LLM by DeepMind and verified in Lean 4. The author takes full responsibility for correctness of the proof.

\printbibliography[heading=bibintoc,title={References}]

\end{document}